\documentclass{amsart}

\usepackage{amssymb}
\usepackage{amsmath}
\usepackage{latexsym}
\usepackage{euscript}
\usepackage{enumerate}
\usepackage{color}

{\left[\begin{smallmatrix}}%            begin code
{\end{smallmatrix}\right]}%

\newcommand{\setdef}[2]{\left\{ #1 :\, \vphantom{#1} #2 \right\}}

   \def\sH{{\mathfrak H}}

      \def\sR{{\mathfrak R}}

      \def\dC{{\mathbb C}}

   \def\dN{{\mathbb N}}

\def\cA{{\EuScript A}}

\def\bm\chi{\mbox{\boldmath$\chi$}}

\def\ker{{\rm ker\,}}
\def\ran{{\rm ran\,}}

\def\dom{{\rm dom\,}}

\def\mul{{\rm mul\,}}

\def\la{\lambda}

\def\dim{{\rm dim\,}}
\def\diag{{\rm diag\,}}

\let\xker=\ker \def\ker{{\xker\,}}
\def\spn{{\rm span\,}}

\newcommand{\veps}{\varepsilon}
\newcommand{\C}{\mathbb{C}}
\newcommand{\N}{\mathbb{N}}

\newtheorem{theorem}{Theorem}[section]
\newtheorem{proposition}[theorem]{Proposition}
\newtheorem{corollary}[theorem]{Corollary}
\newtheorem{lemma}[theorem]{Lemma}
\newtheorem{definition}[theorem]{Definition}
\theoremstyle{definition}

\newtheorem{remark}[theorem]{Remark}

\numberwithin{equation}{section}

\makeatletter
\newcommand{\superimpose}[2]{%
  {\ooalign{$#1\@firstoftwo#2$\cr\hfil$#1\@secondoftwo#2$\hfil\cr}}}
\makeatother

\begin{document}

\title{Linear relations and their singular chains}

\author{Thomas Berger}
\address{Institut f\"{u}r Mathematik\\
Universit\"at  Paderborn\\
Warburger Strasse 100, 33098 Paderborn\\
Germany}
\email{thomas.berger{\scriptsize @}math.upb.de}
\author{Henk de Snoo}
\address{Bernoulli Institute for Mathematics, Computer Science
and Artificial Intelligence\\
University of Groningen \\
P.O.\ Box 407, 9700 AK Groningen \\
Nederland}
\email{h.s.v.de.snoo{\scriptsize @}rug.nl}
\author{Carsten Trunk}
\address{Institut f\"{u}r Mathematik \\
Technische Universit\"{a}t Ilmenau\\
Weimarer Stra{\ss}e~25, 98693~Ilmenau \\
Germany}
\email{carsten.trunk{\scriptsize @}tu-ilmenau.de}
\author{Henrik Winkler}
\address{Institut f\"{u}r Mathematik \\
Technische Universit\"{a}t Ilmenau\\
Weimarer Stra{\ss}e~25, 98693~Ilmenau \\
Germany}
\email{henrik.winkler{\scriptsize @}tu-ilmenau.de}

\dedicatory{Dedicated to our friend  Vladimir Derkach on the occasion of
his seventieth birthday}


\begin{abstract}
 Singular chain spaces for linear relations in linear spaces play a fundamental
role in the decomposition of linear relations in finite-dimensional spaces.
In this paper singular chains and singular chain spaces are discussed in detail
for not necessarily finite-dimensional linear spaces. This leads to an identity
characterizing a singular chain space in terms of root spaces.
The so-called proper eigenvalues of a linear relation play an important role
in the finite-dimensional case.
\end{abstract}

\maketitle

\section{Introduction}\label{Sec:intro}

Let $A$ be a linear relation in a linear space $\sH$.
In this paper the singular chain space $\sR_c(A)$ of $A$ is
 considered in detail. In the nontrivial case, such subspaces can only occur
for relations that are not graphs of linear operators.
The singular chain space $\sR_c(A)$ plays a fundamental role
in the decomposition of a linear relation $A$
when the linear space $\sH$ is finite-dimensional;
see \cite{BSTWnew}, \cite{BTW16},  and \cite{SandDeSn05}.

\medskip
Linear relations in linear spaces go back to \cite{Arens};
see also \cite{BHS} and \cite{Cross}.
Let $\sH$ be a possibly infinite-dimensional linear space over $\dC$.
A linear relation $A$ in $\sH$ is defined
as a linear subspace of the product $\sH \times \sH$.
The usual notions for a linear relation $A$ in $\sH$ are defined as follows:
\[
\begin{split}
\dom A&= \{x \in \sH :\, \exists\, y \in \sH \mbox{ with } (x,y) \in A\},\, \,\rm{domain}, \\
\ker A&= \{ x \in \sH :\, (x,0) \in A\}, \rm{kernel},  \\
\ran A&= \{ y \in \sH :\, \exists\, x \in \sH \mbox{ with } (x,y) \in A\}, \,\, \rm{range}, \\
\mul A &= \{ y \in \sH:\, (0,y) \in A\}, \,\, \rm{multivalued\,\, part},  \\
A^{-1} &= \{(y,x) \in \sH \times \sH:\, (x,y) \in A\}, \,\, \rm{inverse}. \\
\end{split}
\]
Furthermore, for a linear relation $A$
and a complex number $\lambda \in \dC$ one defines
the following operations:
\[
\begin{split}
 A-\lambda&=A-\lambda I=\{(x,y-\lambda x) \in \sH\times \sH :\, (x,y) \in A\},\\
 \lambda A &= \{ (x,\lambda y) \in \sH\times \sH :\, (x,y) \in A \}.   \\
 \end{split}
\]
The product of  linear relations $A$ and $B$ in $\sH$ is given by
\[
AB=\{ (x,y) \in \sH \times \sH:\, \exists\, z \in \sH \mbox{ with } (x,z) \in B, (z,y) \in A\},
\]
and note that
\[
A^k =A A^{k-1}, \quad k\geq 1, \quad \mbox{where} \quad A^0=I.
\]

\medskip
This paper is concerned with root spaces for
a linear relation $A$ in a linear space $\sH$ over $\dC$.
The  \emph{root spaces} $\mathfrak{R}_{\lambda}(A)$ of $A$ for
$\lambda \in \dC\cup \{\infty\}$  are linear subspaces of $\sH$ defined by
\begin{equation}\label{rootsp}
 \mathfrak{R}_{\lambda}(A)=
 \bigcup_{i=1}^{\infty} \, \ker(A-\lambda)^{i},\quad \lambda \in \dC,
 \quad \mbox{and} \quad
\mathfrak{R}_{\infty}(A)=\bigcup_{i=1}^{\infty} \, \mul A^{i}.
\end{equation}
Let $\lambda \in \dC$, then $x \in \sR_\lambda(A)$ if and only if
 for some $n \in \dN$ there exists a chain of elements of the form
\begin{equation}\label{jchain}
 (x_n,x_{n-1}+\lambda x_n ),
 (x_{n-1},x_{n-2}+\lambda x_{n-1} ),
 \dots,
 (x_{2}, x_{1}+\lambda x_{2}),
 (x_{1},\lambda x_{1} ) \in A,
\end{equation}
and such that $x=x_n$, the ``endpoint'' of \eqref{jchain}. The chain in \eqref{jchain}
is said to be a
\textit{Jordan chain} for $A$ corresponding to the eigenvalue $\la \in \dC$;
note that for all $1 \leq i \leq n$ one has $(x_i,0) \in (A-\lambda)^i$.
 In fact, if $x \in \ker (A-\lambda)^n \setminus \ker (A-\lambda)^{n-1}$,
 then $x_1, \dots, x_{n-1}$
are linearly independent and satisfy
$x_i \in \ker (A-\lambda)^i \setminus \ker (A-\lambda)^{i-1}$, $1 \leq i \leq n-1$.
Likewise, $y \in \sR_\infty(A)$ if and only if
 for some $m \in \dN$ there exists a chain of elements of the form
\begin{equation}\label{Jane}
   (0,y_1), \ (y_1,y_2), \ \dots \dots \dots , (y_{m-2}, y_{m-1}), \ (y_{m-1},y_m) \in A,
 \end{equation}
and such that $y=y_m$, the ``endpoint'' of \eqref{Jane}. The chain in \eqref{Jane}
is said to be a \textit{Jordan chain} for $A$ corresponding to the eigenvalue $\infty$;
note that for all $1 \leq i \leq m$ one has  $(0,y_i) \in A^i$.
 In fact, if $y \in \mul A^m \setminus \mul A^{m-1}$, then $y_1, \dots, y_{m-1}$
are linearly independent and satisfy
$y_j \in \mul A^j \setminus \mul A^{j-1}$, $1 \leq j \leq m-1$.

\begin{definition}\label{rca}
Let $A$ be a linear relation in $\sH$.
Then the singular chain space $ \sR_c(A)$ is defined by
\begin{equation}
\sR_c(A)=\sR_0(A) \cap \sR_\infty(A).
\end{equation}
\end{definition}

Note that $u \in \sR_c(A)$ if and only if for some $k \in \dN$
there is a chain of elements of the form
\begin{equation}\label{schain}
  (0,u_k), \ (u_{k} , u_{k-1}), \ \dots \dots \dots , (u_{2}, u_{1}), \ (u_{1},0) \in A,
\end{equation}
and such that $u=u_l$ for some $1 \leq l \leq k$.
The chain in \eqref{schain} is said to be a \textit{singular chain} for $A$.
It is clear from \eqref{schain} that  $\sR_c(A) \subset \dom A \cap \ran A$, and that
$\sR_c(A) \neq \{0\}$ implies that $\ker A \cap \sR_c(A)$ and  $\mul A \cap \sR_c(A)$
are non-trivial.

\medskip
In this paper the singular chain space $\sR_c(A)$
will be characterized by means of different
combinations of the root spaces. As a consequence, one sees that
\[
\sR_c(A) \subset \sR_\lambda(A)
\]
for all $\lambda \in \dC \cup \{\infty\}$ and, hence,  if $\sR_c(A)$ is non-trivial,
all of $\dC \cup \{\infty\}$ consists of eigenvalues of $A$.
However, in this case, the eigenvalues $\lambda \in \dC \cup \{\infty\}$
 for which $\sR_c(A) \neq \sR_\lambda(A)$,
 i.e., the proper eigenvalues, deserve special attention.
The main features are Theorem \ref{rootruht} and  Theorem \ref{Lem:Multsum}
 for the singular chain spaces,
and Theorem \ref{Prop:FinSpec-KronForm} for the proper eigenvalues.
In the finite-dimensional case,  Corollary \ref{lood1},
as a consequence of the last result,
gives some fundamental results for the proper eigenvalues.
 For the convenience of the reader the present paper
has been made self-contained by borrowing
a few arguments from \cite{SandDeSn05}.
The various chain spaces are illustrated by
the case of matrix pencils in Section  \ref{pencil}.
In a further publication \cite{BSTWnew}
the results in the present paper will be used to give a detailed study
of the structure of linear relations
in finite-dimensional spaces, extending \cite{SandDeSn05}.

\section{Some transformation results involving chains}

This section contains some formal transformation results concerning chains
related to a linear relation $A$ in a linear space $\sH$.

\begin{lemma}\label{CHAin}
Let $A$ be a linear relation in $\sH$ and let $\lambda \in \dC$.
Then the  chain
\begin{equation}\label{schain+-}
 (0,x_{1}), (x_{1},x_{2}), \dots ,(x_{n-2},x_{n-1})  ,(x_{n-1},x_{n}) \,\, \mbox{ in } \,\, A-\lambda,
\end{equation}
 is transformed by
 \begin{equation}\label{zzkk}
z_m:=\sum_{i=1}^m \begin{pmatrix} n-i-1 \\ n-m-1 \end{pmatrix} (-\la)^{m-i}x_i,
\quad m=1,\ldots, n-1,
\end{equation}
into a chain
\begin{equation}\label{zk8-}
(0,z_1), (z_1,z_2), \ldots, ,(z_{n-2},z_{n-1}) ,  (z_{n-1},x_n) \,\, \mbox{ in } \,\, A.
\end{equation}
 Moreover,
\begin{equation}\label{zk9}
\spn\{z_1,\ldots, z_{n-1}\}=\spn\{x_1,\ldots, x_{n-1}\}.
\end{equation}
\end{lemma}

\begin{proof}
For later use, note that the assumption \eqref{schain+-} is equivalent to
\begin{equation}\label{schain+-m}
 (0,x_{1}), (x_{1},x_{2}+\lambda x_1), \dots ,  (x_{n-2},x_{n-1}+\lambda x_{n-2}),
 (x_{n-1},x_{n}+\lambda x_{n-1}) \,\, \mbox{ in } \,\, A.
\end{equation}
It follows from \eqref{zzkk} that $z_1=x_1$ and that
$z_{n-1}=\sum_{i=1}^{n-1}  (-\la)^{n-1-i}x_i$.
Hence, it is clear that $(0,z_1) \in A$, and, by working backwards
in \eqref{schain+-m}, one also sees that
\[
 (z_{n-1},x_n)
=(x_{n-1} -\lambda x_{n-2} +\lambda^2 x_{n-3}+\dots+(-\lambda)^{n-2}x_1, x_n) \in A.
\]
Thus it remains to verify that $(z_m, z_{m+1}) \in A$
for $1 \leq m \leq n-2$.
For this purpose, define the elements $\tilde z_m$, $1 \leq m \leq n-2$, by
\[
\tilde z_m
:=\sum_{i=1}^m \begin{pmatrix} n-i-1 \\ n-m-1 \end{pmatrix} (-\la)^{m-i}(x_{i+1}+\lambda x_i),
\]
and it follows from \eqref{schain+-m}  that $(z_m, \tilde z_m) \in A$, $1 \leq m \leq n-2$.
A calculation shows for $1 \leq m \leq n-2$ that
\[
\begin{split}
\tilde z_m
&=\sum_{i=1}^m \begin{pmatrix} n-i-1 \\ n-m-1 \end{pmatrix} (-\la)^{m-i} x_{i+1}
-\sum_{i=1}^m \begin{pmatrix} n-i-1 \\ n-m-1 \end{pmatrix} (-\la)^{m-i+1} x_i  \\
&=\sum_{i=2}^{m+1} \begin{pmatrix} n-(i-1)-1 \\ n-m-1 \end{pmatrix} (-\la)^{m-i+1} x_{i}
-\sum_{i=1}^m \begin{pmatrix} n-i-1 \\ n-m-1 \end{pmatrix} (-\la)^{m-i+1} x_i  \\
&=\sum_{i=2}^{m} \left[ \begin{pmatrix} n-(i-1)-1 \\ n-m-1 \end{pmatrix} -
\begin{pmatrix} n-i-1 \\ n-m-1 \end{pmatrix} \right](-\la)^{m-i+1} x_{i} \\
& \hspace{6cm}+ x_{m+1} -\begin{pmatrix} n-2 \\ n-m-1 \end{pmatrix} (-\lambda)^m x_1\\
&=\sum_{i=2}^{m}  \begin{pmatrix} n-i-1 \\ n-(m+1)-1 \end{pmatrix} (-\la)^{m-i+1} x_{i}
+ x_{m+1} -\begin{pmatrix} n-2 \\ n-m-1 \end{pmatrix} (-\lambda)^m x_1\\
&=\sum_{i=1}^{m+1}  \begin{pmatrix} n-i-1 \\ n-(m+1)-1 \end{pmatrix} (-\la)^{m-i+1} x_{i}
-cx_1, \quad
c =\begin{pmatrix} n-1 \\ n-m-1 \end{pmatrix} (-\lambda)^m.
\end{split}
\]
As a consequence of this calculation and \eqref{zzkk} one sees that
$\tilde z_m=z_{m+1}-cx_1$ with the constant $c$ as indicated.
One concludes that
\[
 (z_m, z_{m+1})=(z_m, \tilde z_m) + c(0,x_1) \in A
\]
for $1 \leq m \leq n-2$.
Hence, the proof of \eqref{zk8-} is complete for the case $n \geq 3$.
The remaining case is trivial.

Finally, observe that the $(n-1) \times (n-1)$ matrix of the transformation
in \eqref{zzkk} is triangular with $1$'s on the diagonal.
Hence, the equality \eqref{zk9} holds.
\end{proof}

\begin{lemma}\label{Lem:Chaintrafo}
Let $A$ be a linear relation in $\sH$ and let $\lambda \in \dC$.
Then the chain
\begin{equation}\label{HH1}
(x_0,x_1), (x_1,x_2),\ldots, (x_{k-1},x_k), (x_k,0)\,\, \mbox{ in } \,\,A-\lambda
\end{equation}
 is transformed by
\begin{equation}\label{HH3}
z_m:=\sum_{i=0}^m \begin{pmatrix} m \\ i \end{pmatrix} \la^{m-i}x_i,
\quad m \geq 0,
\end{equation}
with $x_{i}=0$ for $i>k$, into a chain
\begin{equation}\label{HH4}
 (z_0,z_1), (z_1,z_2), (z_2,z_3), \dots   \,\, \mbox{ in } \,\, A.
\end{equation}
 \end{lemma}

\begin{proof}
For later use, note that the assumption \eqref{HH1} with the additional
convention $x_{i}=0$ for $i>k$ is equivalent to
\begin{equation}\label{HH2}
(x_i,\la x_i +x_{i+1})\in A,\quad i \geq 0.
\end{equation}
It follows from \eqref{HH3} that $z_0=x_0$, $z_1=\la x_0+x_1$.
Hence it is clear that  $(z_0,z_1)\in A$.
Thus it remains to verify $(z_m,z_{m+1}) \in A$ for $m \geq 1$.
For this purpose,  define the elements $\tilde z_m$, $m \geq 1$, by
\[
\tilde z_m
:=\sum_{i=0}^m \begin{pmatrix} m \\ i \end{pmatrix} \la^{m-i}(\la x_i +x_{i+1}),
 \]
and it follows from \eqref{HH2} that  $(z_m,\tilde z_m)\in A$ for $m \geq 1$.
A calculation shows for  $m \geq 1$  that
\begin{eqnarray*}
\tilde z_m&=&\sum_{i=0}^m \begin{pmatrix} m \\ i \end{pmatrix} \la^{m+1-i}x_i+
\sum_{i=0}^m \begin{pmatrix} m \\ i \end{pmatrix} \la^{m-i}x_{i+1}\\
&=&\sum_{i=0}^m \begin{pmatrix} m \\ i \end{pmatrix} \la^{m+1-i}x_i+
\sum_{i=1}^{m+1} \begin{pmatrix} m \\ i-1 \end{pmatrix} \la^{m+1-i}x_{i}\\
&=&\la^{m+1}x_0+\sum_{i=1}^m \left[\begin{pmatrix} m \\ i \end{pmatrix}
+\begin{pmatrix} m \\ i-1 \end{pmatrix}\right] \la^{m+1-i}x_i+ x_{m+1}\\
&=&\la^{m+1}x_0+\sum_{i=1}^m \begin{pmatrix} m+1 \\ i \end{pmatrix} \la^{m+1-i}x_i+ x_{m+1}\\
&=&\sum_{i=0}^{m+1} \begin{pmatrix} m+1 \\ i \end{pmatrix} \la^{m+1-i}x_i.
\end{eqnarray*}
As a consequence of this calculation and \eqref{HH3} one sees that $\tilde z_m=z_{m+1}$.
One concludes that $(z_m, z_{m+1})\in A$ for $m \geq 1$. Hence,  the proof
of \eqref{HH4} is complete.
\end{proof}

\begin{remark}\label{transfo}
The transformation result
in Lemma~\ref{Lem:Chaintrafo} can be written in the following way.
For any $s \geq 0$ one has
\begin{equation}\label{eq:trafo-Bkm}
 [z_0,\ldots,z_s]^\top= C_{s,k}(\la)
 \, [x_0,\ldots,x_k]^\top,\,
\end{equation}
where $C_{s,k}(\lambda)$
is an $(s+1)\times(k+1)$ matrix:
 \begin{equation}\label{Bnk}
C_{s,k}(\lambda)
=(c_{m,i})_{m=0,i=0}^{s,k},
\end{equation}
whose coefficients $c_{m,i}$ are given by
\[
c_{m,i}=\begin{pmatrix} m \\ i \end{pmatrix}\la^{m-i} \mbox{ if } 0\leq i\le m \leq s,
\mbox{ and } c_{m,i}=0 \mbox{ if }  m < i \leq k.
\]
Thus, for $s \geq k$ one has
\[
\large{C_{s,k}(\la)} =
\begin{tiny}
\begin{pmatrix}
\vspace{0.1cm}
1 &0 & 0 &\dots & \dots & 0 & 0 \\
\vspace{0.1cm}
%\binom{1}{0}
\lambda &1 &0 &\dots &\dots & 0 & 0 \\
\vspace{0.1cm}
%\binom{2}{0}
\lambda^2 & \binom{2}{1}\lambda  &1 &\dots &\dots & 0 & 0 \\
\vspace{0.1cm}
\dots&\dots& \dots &\dots& \dots &\dots &\dots\\
\vspace{0.1cm}
\dots&\dots& \dots &\dots& \dots &\dots &\dots \\
\vspace{0.1cm}
%\binom{k-1}{0}
\lambda^{k-1} & \,\,\binom{k-1}{1} \lambda^{k-2} & \binom{k-1}{2} \lambda^{k-3}&\dots & \dots
& 1  & 0\\
\vspace{0.1cm}
%\binom{k}{0}
\lambda^{k} & \,\,\binom{k}{1} \lambda^{k-1} &\binom{k}{2} \lambda^{k-2}&\dots & \dots
&\binom{k}{k-1} \lambda & 1\\
\vspace{0.1cm}
%\binom{k+1}{0}
\lambda^{k+1} & \binom{k+1}{1} \lambda^{k} &\binom{k+1}{2} \lambda^{k-1} &\dots & \dots &
\binom{k+1}{k-1} \lambda^2  &\binom{k+1}{k}\lambda \\
\vspace{0.1cm}
\dots&\dots& \dots &\dots& \dots &\dots &\dots\\
\vspace{0.1cm}
\dots&\dots& \dots &\dots& \dots &\dots &\dots \\
\vspace{0.1cm}
%\binom{s-1}{0}
\lambda^{s-1} & \binom{s-1}{1} \lambda^{s-2} & \binom{s-1}{2} \lambda^{s-3} &\dots & \dots
& \binom{s-1}{k-1} \lambda^{s-k}
& \binom{s-1}{k} \lambda^{s-k-1} \\
%\binom{s}{0}
\lambda^{s} & \binom{s}{1} \lambda^{s-1} & \binom{s}{2} \lambda^{s-2} &\dots & \dots
&\binom{s}{k-1} \lambda^{s-k+1} &
\binom{s}{k}\lambda^{s-k}
\end{pmatrix}.
\end{tiny}
\]
\end{remark}
\vspace{0.2cm}

\section{Identities for root spaces}

Let $A$ be a linear relation in a linear space $\sH$.
There are a number of useful algebraic identities for the root spaces in \eqref{rootsp}.
First of all there is a simple identity involving a shift of the parameter.
It is clear from the definition in \eqref{rootsp} that for all $\lambda, \mu \in \dC$:
\begin{equation}\label{hen22}
 \sR_\la(A) = \sR_{\la-\mu}(A-\mu).
\end{equation}
In particular, one has for all $\lambda \in \dC$:
\begin{equation}\label{rlm}
 \sR_0(A-\lambda)=\sR_\lambda(A).
\end{equation}
The following lemma shows some invariance properties
for the root space at $\infty$.
These results go back to \cite[Lemma 2.3]{SandDeSn05};
the proof is included for completeness.

\begin{lemma}\label{henny}
Let $A$ be a linear relation in $\sH$. Then for all $\lambda \in \dC$
\begin{equation}\label{hen222}
 \sR_\infty(A)=\sR_\infty(A-\lambda),
\end{equation}
and for all $\lambda \in \dC \setminus \{0\}$
\begin{equation}\label{hen2222}
 \sR_\infty(A)=\sR_\infty(\lambda A).
\end{equation}
\end{lemma}

\begin{proof}
First, it will be shown that
$\sR_\infty(A-\lambda) \subset \sR_\infty(A)$ for any $\lambda \in \dC$.
To see this, let $x \in \sR_\infty(A-\lambda)$.
Then $x=x_n$ for some chain of the form \eqref{schain+-}.
By Lemma \ref{CHAin} it follows that $x=z_n$ for some chain of the form \eqref{zk8-}.
In other words, one concludes that $x \in \sR_\infty(A)$. Thus
it follows that $\sR_\infty(A-\lambda) \subset \sR_\infty(A)$ for any $\lambda \in \dC$.
Now it is clear that
\[
 \sR_\infty(A)=\sR_\infty(A-\lambda+\lambda) \subset \sR_\infty(A-\lambda),
\]
where the inclusion is obtained by applying the earlier observation
(with $A$ replaced by $A-\lambda$ and $\lambda$ replaced by $-\lambda$).
Therefore, \eqref{hen222} has been shown.

Next it will be shown that $\sR_\infty(A) \subset \sR_\infty(\lambda A)$
for $\lambda \in \dC \setminus \{0\}$.
To see this, let $x \in \sR_\infty(A)$.
Then $x=x_n$ for some chain of the form
\begin{equation*}
 (0,x_{1}), (x_{1},x_{2}), \dots ,(x_{n-1},x_{n}) \in A.
\end{equation*}
Then, clearly,
\[
 (0, \lambda x_1), (\lambda x_1, \lambda^2 x_2), \dots\dots,
 (\lambda^{n-1}x_{n-1}, \lambda^n x_n) \in \lambda A,
\]
and this shows that $\lambda^n x_n \in \sR_\infty(\lambda A)$,
so that also $x_n \in \sR_\infty(\lambda A)$.
Thus it follows that $\sR_\infty(A) \subset \sR_\infty(\lambda A)$
for $\lambda \in \dC \setminus \{0\}$.
Now observe that
\[
 \sR_\infty(\lambda A)
 \subset \sR_\infty \left( \frac{1}{\lambda} \lambda A\right)=\sR_\infty(A),
 \quad \lambda \in \dC \setminus \{0\},
\]
where the inclusion is obtained by applying the earlier observation
(with $A$ replaced by $\lambda A$ and $\lambda$ replaced by $1/\lambda$).
Therefore, \eqref{hen2222} has been shown.
\end{proof}

The following simple identity, for a linear relation $A$ in $\sH$
and $\lambda \in \dC \setminus \{0\}$,
\begin{equation}\label{ide}
 (A-\lambda)^{-1}=-\frac{1}{\la}-\frac{1}{\la^2}\left(A^{-1}-\frac{1}{\la}\right)^{-1},
\end{equation}
is easily verified; cf. \cite{BHS}. By taking multivalued parts in this identity,
and noting that multivalued parts are shift-invariant, one sees that
\[
 \ker (A-\la) = \ker \left(A^{-1}-\tfrac{1}{\la}\right), \quad \la \in \dC \setminus \{0\}.
\]
This last  identity between the kernels
can be further extended to the root spaces; cf. \cite[Proposition 2.4]{SandDeSn05}.
The simple proof is included as an illustration of Lemma \ref{henny}.

\begin{lemma}\label{rlm0}
Let $A$ be a linear relation in $\sH$. Then for all $\lambda \in \dC \setminus \{0\}$:
\begin{equation}\label{popo}
\sR_{\la}(A)=\sR_{\la^{-1}}(A^{-1}).
 \end{equation}
Moreover,
\begin{equation}\label{hen1}
 \sR_{0}(A)=\sR_{\infty}(A^{-1}),
 \quad
 \sR_{\infty}(A)=\sR_{0}(A^{-1}).
\end{equation}
\end{lemma}

\begin{proof}
The identities in \eqref{hen1} follow directly
from the definition and $\ker A=\mul A^{-1}$.
For $\lambda \in \dC\setminus \{0\}$ it is a consequence
of \eqref{hen222},  \eqref{hen2222}, and \eqref{ide} that
\begin{equation}\label{hola}
 \sR_\infty((A-\lambda)^{-1})
 =\sR_\infty \left(\left(A^{-1}-\frac{1}{\la}\right)^{-1} \right).
\end{equation}
Hence, for  $\lambda \in \dC\setminus \{0\}$, one sees from this
\[
\begin{split}
\sR_\lambda(A) & \stackrel{\eqref{rlm}}{=}   \sR_0(A-\lambda)
\stackrel{\eqref{hen1}}{=} \sR_\infty((A-\lambda)^{-1}) \\
& \stackrel{\eqref{hola}}{=}
\sR_\infty \left(\left(A^{-1}-\frac{1}{\la}\right)^{-1} \right)
\stackrel{\eqref{hen1}}{=}
\sR_0 \left(A^{-1}-\frac{1}{\la}\right)
\stackrel{\eqref{rlm}}{=}  \sR_{\la^{-1}}(A^{-1}),
\end{split}
\]
which gives \eqref{popo}.
\end{proof}

\section{Identities for singular chain spaces}

Let $A$ be a linear relation in a linear space $\sH$.
Recall from  Definition \ref{rca}  that
$\sR_c(A)=\sR_0(A) \cap \sR_\infty(A)$.
In Lemma \ref{henny} it has been shown that
$\sR_\infty(A)$ is shift-invariant.
In fact, it can be shown that also the space
$\sR_c(A)$ in Definition \ref{rca} is invariant
under translations of the relation $A$.

\begin{lemma}\label{hannon}
Let $A$ be a linear relation in $\sH$.
For any $\lambda \in \dC$ one has
\begin{equation}\label{shiftrc}
 \sR_{c}(A)=\sR_{c}(A-\lambda).
\end{equation}
\end{lemma}

\begin{proof}
It will be shown that
$\sR_c(A-\lambda) \subset \sR_{c}(A)$, $\lambda \in \dC$.
To see this last inclusion, let $x\in\sR_{c}(A-\lambda)$.
Then the element $x$ is some entry  of an element
in a chain of the form
\begin{equation*}\label{schain+}
 (0,x_{1}), (x_{1},x_{2}), \dots ,(x_{s-1},x_{s}), (x_{s},0)
 \,\, \mbox{ in } \,\,A-\lambda.
\end{equation*}
By Lemma~\ref{CHAin} the transformation \eqref{zzkk}
produces a chain which satisfies
\begin{equation*}
(0,z_1), (z_1,z_2), \ldots, (z_{s-1},z_s), (z_s,0)\in A,
\end{equation*}
and, in addition,
\begin{equation*}
\spn\{x_1,\ldots, x_s\}=\spn\{z_1,\ldots, z_s\}.
\end{equation*}
By definition, each $z_i \in \sR_c(A)$ and, hence,
$x \in \sR_c(A)$. Therefore it follows that
$\sR_c(A-\lambda) \subset \sR_{c}(A)$, $\lambda \in \dC$.
Now it is clear that
\[
 \sR_c(A)=\sR_c(A-\lambda+\lambda) \subset \sR_c(A-\lambda),
\]
where the inclusion is obtained by applying the earlier observation (with $A$
replaced by $A-\lambda$ and $\lambda$ replaced by $-\lambda$).
Therefore, \eqref{shiftrc} has been shown.
\end{proof}

By Lemma~\ref{hannon}, if $\sR_c(A)=\{0\}$, one sees that for $x \in \ker (A-\lambda)^n$
the chain in \eqref{jchain} with $x=x_n$ is uniquely determined.
Likewise, for $y \in \mul A^m$ the chain in \eqref{Jane}
with $y=y_m$ is uniquely determined if  $\sR_c(A)=\{0\}$. Furthermore, one is able
to characterize $\sR_c(A)$ in a different way,
involving the root space at $\lambda \in \dC$
rather than the one  at $\lambda=0$.

\begin{proposition}\label{Lem:SingChainMani}
Let $A$ be a linear relation in $\sH$. For all $\lambda\in\dC$ one has
\begin{equation}\label{il+}
 \sR_c(A) = \sR_\lambda(A) \cap \sR_{\infty}(A).
\end{equation}
\end{proposition}

\begin{proof}
Observe that by \eqref{shiftrc}
one has that
\[
\sR_c(A)=\sR_c(A-\lambda)=\sR_0(A-\lambda) \cap \sR_\infty(A-\lambda)
 =\sR_\lambda(A) \cap \sR_\infty(A),
\]
where the last equality was obtained by using the identities
 \eqref{rlm} and \eqref{hen222}.
\end{proof}

It was shown in \cite{SandDeSn05} that
$\sR_c(A)$ is non-trivial if and only if
$\sR_{\la}(A)\cap\sR_{\mu}(A)$ is non-trivial.
In fact, Proposition \ref{Lem:SingChainMani}
is the stepping stone for the following general result,
namely,
the calculation of the intersection of $\sR_{\la}(A)$ and $\sR_{\mu}(A)$.

\begin{theorem}\label{rootruht}
Let $A$ be a linear relation in $\sH$.
Then one has for $\la, \mu\in\dC\cup\{\infty\}$ with $\la \ne \mu$ that
\[
\sR_{c}(A)=\sR_{\la}(A)\cap\sR_{\mu}(A).
\]
\end{theorem}

\begin{proof}
Due to Proposition \ref{Lem:SingChainMani} it suffices to consider the case
where $\lambda, \mu \in \dC$ and $\lambda \ne \mu$.
With these restrictions define the linear relation $\tilde A$ by
\begin{equation}\label{awurst}
\tilde A=(A-\la)^{-1}-(\mu-\la)^{-1},
\end{equation}
so that the following identities are clear
\begin{equation}\label{aawurst}
\begin{split}
  \sR_{\infty}(\tilde A)
& \stackrel{\eqref{awurst}}{=} \sR_\infty((A-\la)^{-1}-(\mu-\la)^{-1}) \\
& \stackrel{\eqref{hen222}}{=} \sR_\infty((A-\lambda)^{-1})
 \stackrel{\eqref{hen1}}{=} \sR_0(A-\lambda)
 \stackrel{\eqref{rlm}}{=} \sR_{\la}(A).
\end{split}
\end{equation}
By Proposition~\ref{Lem:SingChainMani}
one may write  $\sR_c(\tilde A)=\sR_0(\tilde A) \cap \sR_\infty(\tilde A)$ as
 \begin{equation}\label{B}
 \sR_c(\tilde A)=\sR_{-(\mu-\la)^{-1}}(\tilde A) \cap \sR_\infty(\tilde A),
\end{equation}
since $\lambda \neq \mu$. Now observe that
\begin{equation}\label{aaaawurst}
\begin{split}
 \sR_{-(\mu-\la)^{-1}}(\tilde A)
 &\stackrel{\eqref{awurst}}{=}\sR_{-(\mu-\la)^{-1}}((A-\la)^{-1}-(\mu-\la)^{-1})\\
 & \stackrel{\eqref{rlm}}{=} \sR_{0}((A-\la)^{-1})
 \stackrel{\eqref{hen1}}{=} \sR_{\infty}(A-\lambda)
 \stackrel{\eqref{hen222}}{=} \sR_{\infty}(A).
\end{split}
\end{equation}
 Combining  \eqref{B} with \eqref{aawurst} and \eqref{aaaawurst}  leads to
 \begin{equation}\label{BBq}
 \sR_c(\tilde A)=\sR_\lambda(A) \cap \sR_\infty(A)=\sR_c(A),
\end{equation}
by Proposition \ref{Lem:SingChainMani}.
Moreover, observe that
\begin{equation}\label{aaawurst}
\begin{split}
 \sR_{0}(\tilde A)&
 \stackrel{\eqref{awurst}}{=} \sR_0((A-\la)^{-1}-(\mu-\la)^{-1}) \\
 & \stackrel{\eqref{rlm}}{=} \sR_{(\mu-\la)^{-1}}((A-\lambda)^{-1})
 \stackrel{\eqref{popo}}{=} \sR_{\mu-\la}(A-\lambda)
 \stackrel{\eqref{hen22}}{=} \sR_\mu(A).
\end{split}
\end{equation}
Consequently,  a combination of \eqref{BBq},
the definition of $\sR_c(\tilde A)$,
\eqref{aawurst}, and \eqref{aaawurst} leads to
\[
 \sR_c(A)=\sR_c(\tilde A)=\sR_0(\tilde A) \cap \sR_\infty(\tilde A)
 =\sR_{\la}(A)\cap\sR_{\mu}(A). \qedhere
\]
 \end{proof}

The following theorem is the main result in this section.
It will lead to  a final extension of Theorem~\ref{rootruht}.

\begin{theorem}\label{Lem:Multsum}
Let $A$ be a linear relation in $\sH$ and let the sets
\begin{equation}\label{AA}
\{\la_1,\ldots,\la_l\}\subseteq\dC
\quad \mbox{and} \quad
\{\mu\}\subseteq\dC\cup\{\infty\}
\end{equation}
be disjoint.
Assume that
\begin{equation}\label{BB}
x^r\in\sR_{\la_r}(A),  \quad r=1,\ldots,l,
\quad \mbox{and} \quad
\sum_{r=1}^l x^r\in\sR_{\mu}(A).
\end{equation}
Then $x^r\in\sR_c(A)$ for $r=1,\ldots,l$.
\end{theorem}

\begin{proof}
Assume the conditions in \eqref{AA} and
let the elements $x^r$, $1 \leq r \leq l$,
satisfy the conditions in \eqref{BB}.
The proof will be given in two steps:
in the first step the case with $\mu=\infty$ will be considered
and in the second step the case with $\mu \in \dC$
will be considered by a reduction to Step 1.

\medskip\noindent
\textit{Step 1}: Assume that $\mu=\infty$. Then \eqref{BB} reads
 \begin{equation}\label{BBe}
x^r\in\sR_{\la_r}(A),  \quad r=1,\ldots,l,
\quad \mbox{and} \quad
\sum_{r=1}^l x^r\in\sR_{\infty}(A).
\end{equation}
By the first condition in \eqref{BBe}: $x^r\in\sR_{\la_r}(A)$,
there exist $k_r\in\dN$ and chains
\[
(x_0^r,x_1^r), \ldots, (x_{k_r-1}^r,x_{k_r}^r), (x_{k_r}^r,0)
\in A-\la_r, \quad r=1,\ldots ,l,
\]
with $x_0^r=x^r$. Set $x_i^r=0$ for $i>k_r$,
then an application of Lemma~\ref{Lem:Chaintrafo}
shows that for all $m \geq 0$ one has with
\[
 z_m^r=\sum_{i=0}^m \begin{pmatrix} m \\ i \end{pmatrix} \la_r^{m-i}x_i^r,
\]
that $z_0^r=  x_0^r=x^r$ and for all $m \geq 0$
 \[
 (z_m^r, z_{m+1}^r) \in A.
 \]
 It is helpful to introduce the notation
\begin{equation}\label{www}
w_m=\sum_{r=1}^l z_m^r, \quad m \geq 0,
\end{equation}
which leads to
 \[
 w_0=\sum_{r=1}^l x^r \in \sR_\infty(A), \quad (w_m,w_{m+1}) \in A.
\]
Therefore, it is clear that $w_m\in\sR_{\infty}(A)$ for all $m\in\dN \cup \{0\}$.

The  statement
 $x^r\in\sR_\infty(A)$ for $r=1,\ldots,l$,
 will now be retrieved from \eqref{www}
 for suitably many indices $m$. For this purpose,
 recall that with   the $(s+1)\times(k_r+1)$
 matrix $C_{s,k_r}(\la_r)$, as defined in \eqref{Bnk},
one may write for any $s \geq 0$:
 \begin{equation}\label{eq:trafo-Bkm+}
 [z_0^r,\ldots,z_s^r]=[x_0^r,\ldots,x_{k_r}^r]\, C_{s,k_r}(\la_r)^\top.
\end{equation}
 In the present context, choose $s \geq 0$ as
$$
s+1=\sum_{r=1}^l \,(k_r+1),
$$
so that, in particular, each $k_i \leq s$.
Observe that, due to Remark \ref{transfo}, the matrix
\begin{equation}\label{vvv}
W=\left( C_{s,k_1}(\la_1),\, \cdots  \,,C_{s,k_l}(\la_l) \right)
 \end{equation}
is, in fact, an $(s+1)\times(s+1)$
confluent Vandermonde matrix.  This matrix  $W$ is invertible,
since it is well-known that
\begin{equation}\label{detvan}
 \det W =
\prod_{1\le i<j\le l} (\la_i-\la_j)^{(k_i+1)(k_j+1)}.
\end{equation}
A short proof of \eqref{detvan} in a more general setting can be found in \cite{HaG80}.
By means of  \eqref{www},  \eqref{eq:trafo-Bkm+}, and \eqref{vvv}, one may write:
\begin{equation}\label{vvvv}
\begin{split}
 [w_0,\ldots,w_s] & =[z_0^1+\dots+z_0^l, \dots, z_s^1+\dots+z_s^l ] \\
 &= \big[x_0^1,\ldots,x_{k_1}^1,\ldots, x_0^l,\ldots,x_{k_l}^l\big]\, W^\top,
\end{split}
\end{equation}
 and it therefore follows from the invertibility of $W$ that all vectors
\[
x_0^1,\ldots,x_{k_1}^1,\ldots, x_0^l,\ldots,x_{k_l}^l
\]
in \eqref{vvvv} are linear combinations of $w_0,\ldots,w_s$
and, hence, they belong to $\sR_{\infty}(A)$.
Moreover, by Proposition~\ref{Lem:SingChainMani} one finds that
\[
x^r=x_0^r \in\sR_{\la_r}(A)\cap\sR_{\infty}(A)=\sR_{c}(A),\quad r=1,\ldots,l.
\]

\medskip\noindent
\emph{Step 2}: Assume that $\mu\in\dC$.
To reduce this case to the situation in Step 1
the linear relation $\tilde A$ is introduced by
\begin{equation}\label{last}
\tilde A=(A-\mu)^{-1}.
\end{equation}
By the assumptions in \eqref{AA}, it is clear that the sets
\begin{equation}\label{AAA}
\left\{ \frac{1}{\la_1-\mu}, \, \cdots \, , \frac{1}{\la_l-\mu}\right\}
\quad \mbox{and} \quad \{\infty\}
\end{equation}
are disjoint.
In addition, one sees that
\begin{equation*}\label{mula1}
\sR_\mu(A) \stackrel{\eqref{rlm}}{=} \sR_0(A-\mu)
\stackrel{\eqref{hen1}}{=} \sR_\infty((A-\mu)^{-1})
\stackrel{\eqref{last}}{=}
\sR_\infty(\tilde A),
\end{equation*}
and for $1 \leq r \leq l$
\begin{equation*}\label{mula2}
\sR_{\la_r}(A) \stackrel{\eqref{hen22}}{=} \sR_{\la_r-\mu}(A-\mu)
\stackrel{\eqref{popo}}{=}
\sR_{(\la_r-\mu)^{-1}}((A-\mu)^{-1})
\stackrel{\eqref{last}}{=}
\sR_{(\la_r-\mu)^{-1}}(\tilde A).
\end{equation*}
Therefore, by the assumptions in \eqref{BB}, one obtains
\begin{equation}\label{BBB}
x^r\in\sR_{(\la_r-\mu)^{-1}}(\tilde A), \quad r=1,\ldots,l,
 \quad \mbox{and} \quad
\sum_{r=1}^l x^r\in \sR_\infty(\tilde A).
\end{equation}
Hence, according to \eqref{AAA} and \eqref{BBB},
Step 1 may be applied when the linear relation $A$
and the scalars $\lambda_r$, $ r=1,\ldots,l$, are replaced by
\[
 \tilde A \quad \mbox{and} \quad \frac{1}{\lambda_r-\mu}, \,\,  r=1,\ldots,l.
\]
This leads to the conclusion that
 $x^r\in\sR_c(\tilde A)$ for $r=1,\ldots,l$. Finally, observe that
\begin{equation}\label{mula3}
\begin{split}
\sR_c(\tilde A) & \stackrel{\eqref{last}}{=} \sR_c((A-\mu)^{-1})
=\sR_0((A-\mu)^{-1})\cap\sR_\infty((A-\mu)^{-1}) \\
& \stackrel{\eqref{hen1}}{=} \sR_0(A-\mu)\cap\sR_\infty(A-\mu)
 =\sR_c(A-\mu)\stackrel{\eqref{shiftrc}}{=}\sR_c(A).
\end{split}
\end{equation}
Thus it follows that $x^r\in\sR_c(A)$  for $r=1,\ldots,l$.
 \end{proof}

There is a special case of Theorem \ref{Lem:Multsum} worth mentioning:
\[
 x^r\in\sR_{\la_r}(A), \,\, r=1,\ldots,l \quad \mbox{and} \quad \sum_{r=1}^l x^r=0
 \quad \Rightarrow \quad x^r\in\sR_c(A), \,\, r=1,\ldots,l.
\]
To see this, observe that  $\sum_{r=1}^l x^r \in\sR_{\mu}(A)$
for any $\mu \in \dC \cup \{\infty\}$ and, in particular, for
any $\mu \in \dC \cup \{\infty\}$ that does not belong to $\{\la_1,\ldots,\la_l\}\subseteq\dC$.

\medskip

Theorem \ref{Lem:Multsum} leads to Corollary \ref{Thm:ModRc}
below, which is an extension Theorem \ref{rootruht}.

 \begin{corollary}\label{Thm:ModRc}
Let $A$ be a linear relation in $\sH$ and let
\begin{equation}\label{cutRc0}
\{\la_1,\ldots,\la_l\} \quad \mbox{and}  \quad \{\mu_1,\ldots,\mu_m\}
\end{equation}
be two disjoint subsets of $\dC\cup\{\infty\}$. Then
\begin{equation}\label{cutRc}
\sR_c(A)
=\left(\sum_{i=1}^l\sR_{\la_i}( A)\right) \bigcap \left(\sum_{j=1}^m\sR_{\mu_j}( A)\right).
\end{equation}
\end{corollary}

 \begin{proof}
It is clear from Theorem \ref{rootruht} that the left-hand side of \eqref{cutRc}
is contained in the right-hand side. Therefore it suffices to show that
the right-hand side of \eqref{cutRc} is contained in the left-hand side.
Let $x$ be an element of the space on the right-hand side of~\eqref{cutRc},
which means that $x$ can be written as
\[
x=\sum_{i=1}^l y_i=\sum_{j=1}^m z_j \quad \mbox{with} \quad
y_i\in\sR_{\la_i}(A), \,\,z_j\in\sR_{\mu_j}(A).
\]
In the sets in \eqref{cutRc0} at most one element equals $\infty$. If this is the case, let it be
$\lambda_1$ without loss of generality.  In all cases, it is clear that
\[
 \sum_{j=1}^m z_j -\sum_{i=2}^l y_i =y_1 \in \sR_{\la_1}(A).
\]
Therefore, by  Theorem~\ref{Lem:Multsum} it follows that  $z_j \in \sR_c(A)$, $j=1, \dots,   m$, and $y_i \in \sR_c(A)$, $i=2, \dots, l$, and consequently, also $y_1 \in \sR_c(A)$. Thus one sees that $x \in \sR_c(A)$.
One concludes that the right-hand side of \eqref{cutRc} is contained in the left-hand side.
 Thus the identity \eqref{cutRc} has been proved.
\end{proof}

\section{The proper point spectrum of a linear relation}

The number $\lambda \in \dC$ is an \emph{eigenvalue}
of  $A$, if $\ker (A-\la) \neq \{0\}$, and $\infty$ is
an \emph{eigenvalue} of~$A$, if $\mul A \neq \{0\}$.
The usual \emph{point spectrum} $\sigma_p(A)$
is the set of all eigenvalues $\la\in\dC\cup\{\infty\}$ of the relation $A$:
\begin{equation}\label{ook333-}
 \sigma_p(A)=\{ \lambda \in \dC \cup \{\infty\} :\, \text{$\lambda$ is an eigenvalue of $A$} \, \}.
\end{equation}
It follows from \eqref{il+} that
\begin{equation}\label{il++}
\sR_c(A) \subset \sR_\lambda(A), \quad \lambda \in \dC \cup \{\infty\}.
\end{equation}
Thus, by \eqref{il++}, the following implication is trivial:
 \begin{equation}\label{il+++}
\sR_c(A) \neq \{0\} \quad \implies \quad \sigma_p(A) =\dC\cup\{\infty\}.
\end{equation}
Note that the present geometric treatment takes care of
\cite[Proposition 3.2, Corollary 3.3, Corollary 3.4]{SandDeSn05}.
The following definition is based on the inclusion \eqref{il++}.

\begin{definition}\label{Proper!!!}
Let $A$ be a linear relation in a linear space $\sH$.
The \textit{proper point spectrum} $\sigma_{\pi}(A)$
is a subset of the point spectrum $\sigma_{p}(A)$,
defined for $\lambda \in \dC \cup \{\infty\}$ by
\begin{equation}\label{ook333}
   \lambda \in\sigma_{\pi}(A)
   \quad \iff
   \quad
    \sR_{\lambda}(A) \setminus \sR_c(A)  \neq\emptyset.
\end{equation}
The elements in $\sigma_\pi(A)$ are called the
\textit{proper eigenvalues} of $A$.
\end{definition}

Note that if $\sR_c(A)=\{0\}$, then $\sigma_{\pi}(A)=\sigma_{p}(A)$.
 The following result is a direct consequence of Corollary \ref{Thm:ModRc}.

\begin{theorem}\label{Prop:FinSpec-KronForm}
Let $A$ be a linear relation in a linear space $\sH$.
Assume that there exists $k\in\N$ such that
$\lambda_1,\ldots,\lambda_k \in \sigma_{\pi}(A)$
are pairwise distinct proper eigenvalues and let
\[
x_i \in \sR_{\lambda_i}(A) \setminus \sR_c(A),\quad i=1, \ldots, k.
\]
Then the elements $x_1, \ldots, x_k$ are linearly independent in $\sH$.
 \end{theorem}

\begin{proof}
Seeking a contradiction, assume that the elements $x_1, \ldots, x_k$
are linearly dependent.
In fact, assume without loss of generality that
\[
\sum_{i=1}^k c_i x_i=0 \quad \mbox{with} \quad  c_i\in\dC,\,\, i=1,\ldots,k,
\]
and that $c_k \neq 0$. Then, clearly,
\[
    c_k x_{k} =
    -\sum_{i=1}^{k-1} c_i x_i \in \sR_{\lambda_{k}}(A)
    \cap \spn \setdef{\,\sR_{\lambda_i}(A)}{i=1,\ldots,k-1\,}.
\]
Since  $c_k\neq 0$, this implies, by Corollary \ref{Thm:ModRc},
that $ x_{k} \in \sR_c(A)$, a contradiction.
Hence $x_1, \ldots, x_k$ are linearly independent.
 \end{proof}

If the linear space $\sH$ is finite-dimensional,
then there is a bound for the number of  proper eigenvalues
in $\sigma_\pi(A)$, as follows
from Theorem \ref{Prop:FinSpec-KronForm}.

\begin{corollary}\label{lood1}
Let $A$ be a linear relation in a  finite-dimensional space $\sH$. Then
\begin{equation}\label{ook1}
    |\sigma_{\pi}(A) |\le \dim \sH,
\end{equation}
so that, in particular, $\sigma_{\pi}(A)$ consists of finitely many elements.
\end{corollary}

The following result is a simple consequence of
Theorem \ref{Prop:FinSpec-KronForm} and Corollary \ref{lood1};
it goes back to \cite[Lemma 4.5, Theorem 5.1]{SandDeSn05}.

\begin{corollary}
Let $A$ be a linear relation in a finite-dimensional space $\sH$
and assume that $\sR_c(A)=\{0\}$.
Let $\lambda_i \in \sigma_{p}(A)$ be pairwise distinct for $i=1, \ldots, k$,
and let $x_i \in \sR_{\lambda_i}(A)$ be nontrivial.
Then the elements $x_1, \ldots, x_k$ are linearly independent and, consequently,
$|\sigma_{p}(A) |\le \dim \sH$.
\end{corollary}

\medskip
The implication in \eqref{il+++} can be reversed if the space is finite-dimensional;
cf. \cite[Theorem 4.4]{SandDeSn05}.
 A proof is included for completeness.

\begin{proposition}
Let $A$ be a linear relation in $\sH$ and let $\sH$ be finite-dimensional. Then
\begin{equation}\label{AmazonPrime}
\sR_c(A) \neq \{0\} \quad \iff \quad \sigma_p(A) =\dC\cup\{\infty\}.
\end{equation}
\end{proposition}

\begin{proof}
($\Rightarrow$) This is  \eqref{il+++}.

($\Leftarrow$)
Assume that $\sigma_p(A) =\dC\cup\{\infty\}$.
By assumption, $\sH$ is finite-dimensional,
say, $\dim \sH=m$. Therefore, let $\lambda_1, \dots, \lambda_{m+1}$
in $\dC$ be different eigenvalues of $A$
and let  $(x_i, \lambda_i x_{i}) \in A$
with nontrivial $x_i \in \sH$, $1 \leq i \leq m+1$.
 Then $x_1, \dots, x_{m+1}$ are linearly dependent
and there exist $c_i \in \dC$ such that
\[
 \sum_{i=1}^{m+1} c_i x_i=0, \quad \sum_{i=1}^{m+1} |c_i| >0.
\]
Clearly, $c_i x_i\in  \sR_{\lambda_i}(A)$ and one may choose
 $\mu\in\C$ with $\mu\neq \lambda_i$ for $i=1,\ldots,m+1$
 so that $\sum_{i=1}^{m+1} c_i x_i=0 \in \sR_{\mu}(A)$.
 Then Theorem~\ref{Lem:Multsum} implies that $c_i x_i \in\sR_c(A)$
 for all $i=1,\ldots,m+1$, hence $\sR_c(A) \neq \{0\}$.
\end{proof}

\section{Matrix pencils and linear relations}\label{pencil}

Linear relations naturally appear in the study of matrix pencils.
Let $E$ and $F$ be matrices in $\mathbb C^{m\times d}$
and consider the associated matrix pencil $sE-F$,
which is a polynomial matrix in $\dC[s]^{m\times d}$.
Then it is natural to consider the corresponding operator range
\begin{equation}\label{penci}
\cA = \ran
\begin{pmatrix} E\\ F\end{pmatrix},
\end{equation}
so that $\cA$ is a linear relation in the space $\sH=\dC^m$.
There is a close connection in terms of the root spaces and singular chain spaces of $\cA$
and the pencil $sE-F$, which will be described in the following by means of
 the \textit{Kronecker canonical form} for linear matrix pencils,
see e.g.~\cite{BergTren12,BergTren13,G59}.
For this purpose, define for $k\in\mathbb N$ the matrices $N_k$, $K_k$, and $L_k$ by
\[
N_k:=\left(\begin{array}{cccc}
0&&&\\
1&0&&\\
&\ddots&\ddots&\\
&&1&0
\end{array}
\right)\in\mathbb C^{k\times k},
\]
and
\[
K_k:=\left(\begin{array}{cccc}
1 & 0 &&\\
& \ddots& \ddots&\\
&&1&0
\end{array}
\right),\quad
L_k:=\left(\begin{array}{cccc}
0 & 1 &&\\
& \ddots& \ddots&\\
&&0&1
\end{array}
\right)\in \dC^{(k-1)\times k}.
\]
 Likewise, for a multi-index  $\alpha=(\alpha_1,\ldots,\alpha_l)\in\N^l$ with absolute value $|\alpha| = \sum_{i=1}^l \alpha_i$, define
\begin{align*}
N_{\alpha}&:=\diag(N_{\alpha_1},\ldots,N_{\alpha_l})\in\mathbb C^{|\alpha|\times |\alpha|},\\
K_\alpha&:= \diag(K_{\alpha_1},\ldots,K_{\alpha_l}),\quad
L_\alpha:= \diag(L_{\alpha_1},\ldots,L_{\alpha_l})\in\C^{(|\alpha|-l) \times |\alpha|}.
\end{align*}
See, e.g.,~\cite{BTW16} for a discussion of the case
that some or all of the entries of~$\alpha$ are equal to one.

According to Kronecker~\cite{K90},
there exist invertible matrices $W\in\mathbb C^{m\times m}$
and $T\in\dC^{d\times d}$ such that
\begin{align}
\label{kcf}
W(sE-F)T =
\begin{pmatrix} sI_{n_0}-A_0& 0&0&0 \\0& sN_\alpha-I_{|\alpha|}&0&0\\ 0&0& sK_{\veps}-L_{\veps} &0\\ 0&0&0& sK_{\eta}^\top-L_{\eta}^\top\end{pmatrix},
\end{align}
for some $A_0\in\mathbb C^{n_0\times n_0}$ in Jordan canonical form,
and multi-indices $\alpha$, $\veps$, $\eta$, with lengths $n_\alpha$, $n_\veps$, $n_\eta$,
respectively,
%$\alpha\in\N^{n_\alpha}, \veps\in\N^{n_\veps}$ and $\eta\in\N^{n_\eta}$, which are
ordered non-decreasingly. For details, see~\cite[Chapter XII]{G59} or~\cite{KM06}.

The following theorem shows how the singular chain space and the root spaces
together with the proper eigenvalues can be read off from the Kronecker canonical form.
In this sense the Kronecker canonical form can be seen as a canonical form for the linear
relation in~\eqref{penci}.
Vice versa,  this provides a simple
geometric interpretation for the four parts in the Kronecker canonical form.

\begin{theorem}\label{afectado}
Let $E$ and $F$ be matrices in $\mathbb C^{m\times d}$ with Kronecker canonical form \eqref{kcf}
and let the linear relation $\cA$ be defined by \eqref{penci}.
Then the following statements hold:
\begin{enumerate} \def\labelenumi{\rm(\roman{enumi})}
\item
the singular chain space $\sR_c(\cA)$ is given by
\begin{equation*}%\label{eq:Rc_Weimar}
     \sR_c(\cA) = W^{-1} \left( \{0\}^{n_0} \times \{0\}^{|\alpha|}
\times \mathbb C^{|\veps|-n_\veps} \times \{0\}^{|\eta|}\right);
\end{equation*}
\item
 the root space $\sR_\infty(\cA)$ is given by
\begin{equation*}
     \sR_\infty(\cA) = W^{-1} \left( \{0\}^{n_0} \times \mathbb C^{|\alpha|}
\times  \mathbb C^{|\veps|-n_\veps} \times \{0\}^{|\eta|}\right);
\end{equation*}
\item
the root space $\sR_\lambda(\cA)$ for some $\lambda \in \mathbb C$ is given by
\begin{equation*}
     \sR_\lambda(\cA) = W^{-1} \left(\sR_\lambda(A_0) \times \{0\}^{|\alpha|}
\times  \mathbb C^{|\veps|-n_\veps} \times \{0\}^{|\eta|}\right),
\end{equation*}
where $\sR_\lambda(A_0)$ is the root space for $\lambda \in \dC$
of the matrix $A_0$ in
\eqref{kcf};
\item
the proper eigenvalues of $\cA$ are given by
\[
 \sigma_\pi(\cA)=\left\{
 \begin{array}{ll}
 \sigma_p(A_0)\cup\{\infty\},& \mbox{if }\, |\alpha| \neq 0,\\[1ex]
 \sigma_p(A_0),& \mbox{if }\, |\alpha| = 0,
 \end{array}
 \right.
\]
where $\sigma_p(A_0)$ is the point spectrum
of the matrix $A_0$ in
\eqref{kcf}.
\end{enumerate}
\end{theorem}

\begin{proof}
Statements (i) and (ii) are from \cite[Theorem 4.5]{BTW16}.

In order to show (iii), let $x\in\sR_\lambda(\cA)\setminus\{0\}$. Then
there exists a chain of the form~\eqref{jchain} with $x=x_n$.
 By \eqref{penci} there exist $z_1,\ldots,z_{n}\in\C^d$ such that
\begin{equation}\label{eq:FGzi}
     \begin{array}{rcl}
     (x_1,\lambda x_1) &=& (Ez_1, Fz_1),\\
     (x_2, x_1 + \lambda x_2) &=& (Ez_2, Fz_2), \\
     \dots & \dots& \dots\\
     (x_{n-1},x_{n-2}+\lambda x_{n-1}) &=& (Ez_{n-1}, Fz_{n-1}),\\
     (x_n, x_{n-1}+\lambda x_n) &=& (Ez_{n}, Fz_{n}).
     \end{array}
\end{equation}
For $i\in \{1,\ldots,n\}$
define $y_i=T^{-1} z_i$. Partitioning $y_i$ according to the decomposition~\eqref{kcf}:
 \[
y_i = (y_{i,1}^\top, \ldots, y_{i,4}^\top)^\top \quad \mbox{with} \quad  y_{i,1}\in\C^{n_0},
\,\, y_{i,2}\in\C^{|\alpha|}, \,\, y_{i,3}\in\C^{|\veps|}, \,\, y_{i,4}\in\C^{|\eta|-n_\eta},
\]
 one obtains from the first equation in~\eqref{eq:FGzi} that
\begin{equation*}
\begin{array}{rcl}
     Wx_1 &=& WEz_1 =  (WET) T^{-1} z_1 = \begin{pmatrix}
I_{n_0}&0&0&0\\ 0&N_\alpha &0&0\\ 0&0&K_\veps&0\\ 0&0&0&K_\eta^\top\end{pmatrix}
\begin{pmatrix} y_{1,1}\\ y_{1,2}\\ y_{1,3}\\ y_{1,4}\end{pmatrix},\\
&&\\
    \lambda W x_1 &=& WFz_1 =  (WFT) T^{-1} z_1
     =  \begin{pmatrix} A_0
&0&0&0\\ 0&I_{|\alpha|}&0&0\\ 0&0&L_\veps&0\\ 0&0&0&L_\eta^\top\end{pmatrix}
\begin{pmatrix} y_{1,1}\\ y_{1,2}\\ y_{1,3}\\ y_{1,4}\end{pmatrix}.
\end{array}
\end{equation*}
 Therefore, each of the four components leads to an equation, thus
\begin{align*}
 (A_0-\lambda I_{n_0}) y_{1,1} &=0 ,\quad (I_{|\alpha|} -\lambda N_\alpha) y_{1,2} =0,\\
(L_\veps-\lambda K_\veps)  y_{1,3} &=0, \quad (L_\eta^\top-\lambda K_\eta^\top)  y_{1,4} =0.
\end{align*}
Clearly, if $y_{1,1} \neq 0$, $\lambda$ is an eigenvalue
of $A_0$ with eigenvector  $y_{1,1}$.
Moreover, $I_{|\alpha|} -\lambda N_\alpha$ is an invertible matrix and,
invoking~\cite[Lemma~4.1]{BTW16}, one obtains
$$
y_{1,2}=0 \quad \mbox{and} \quad y_{1,4}=0,
$$
hence,
\begin{equation}\label{quisquilloso}
Wx_1 = \begin{pmatrix} y_{1,1}\\ 0\\K_\veps y_{1,3}\\ 0\end{pmatrix}.
\end{equation}
The second equation in~\eqref{eq:FGzi} gives
\begin{equation*}
\begin{array}{rcl}
     Wx_2 &=& WEz_2 =  (WET) T^{-1} z_2 = \begin{pmatrix}
I_{n_0}&0&0&0\\ 0&N_\alpha &0&0\\ 0&0&K_\veps&0\\ 0&0&0&K_\eta^\top\end{pmatrix}
\begin{pmatrix} y_{2,1}\\ y_{2,2}\\ y_{2,3}\\ y_{2,4}\end{pmatrix},\\
&&\\
    W x_1+ \lambda W x_2 &=& WFz_2 =  (WFT) T^{-1} z_2
     =  \begin{pmatrix} A_0
&0&0&0\\ 0&I_{|\alpha|}&0&0\\ 0&0&L_\veps&0\\ 0&0&0&L_\eta^\top\end{pmatrix}
\begin{pmatrix} y_{2,1}\\ y_{2,2}\\ y_{2,3}\\ y_{2,4}\end{pmatrix}.
\end{array}
\end{equation*}
Thanks to \eqref{quisquilloso}, the first, second, and fourth components of the identity
\[
(W x_1+ \lambda W x_2) - \lambda W x_2=Wx_1
\]
 lead to the equations
 $$
(A_0-\lambda I_{n_0}) y_{2,1} =y_{1,1} ,\quad (I_{|\alpha|} -\lambda N_\alpha) y_{2,2} =0, \quad
 (L_\eta^\top-\lambda K_\eta^\top)  y_{2,4} =0.
$$
In the same way as above, one sees that
$$
y_{2,2}=0 \quad \mbox{and} \quad y_{2,4}=0
$$
and, hence,
\begin{equation*}
Wx_2 = \begin{pmatrix} y_{2,1}\\ 0\\K_\veps y_{2,3}\\ 0\end{pmatrix}.
\end{equation*}
The third equation  in~\eqref{eq:FGzi} and, in fact,  all the remaining equations  in~\eqref{eq:FGzi}
are of the same form as the second one. Hence, proceeding in this way, one obtains
for $i=1,\ldots,n$:
\begin{equation*}
     Wx_i =  \begin{pmatrix} y_{i,1}\\ 0\\K_\veps y_{i,3}\\ 0\end{pmatrix}
     \quad\mbox{and}\quad
     (A_0-\lambda I_{n_0}) y_{i,1} =y_{i-1,1},  ,
\end{equation*}
where $y_{0,1}:=0$.
Therefore, one has $y_{i,1}\in \sR_\lambda(A_0)$ for $i=1,\ldots,n$ and this shows
\begin{equation}\label{eq:Rc_Weimar-}
     \sR_\lambda(\cA) \subseteq W^{-1} \left(\sR_\lambda(A_0) \times \{0\}^{|\alpha|}
\times \mathbb C^{|\veps|-n_\veps} \times \{0\}^{|\eta|} \right).
\end{equation}

The reverse inclusion remains to be shown:
\begin{equation}\label{eq:Rc_Weimar}
    W^{-1} \left( \sR_\lambda(A_0) \times \{0\}^{|\alpha|}
\times \mathbb C^{|\veps|-n_\veps} \times \{0\}^{|\eta|}\right)
 \subset \sR_\lambda(\cA).
\end{equation}
If $\lambda$ is not an eigenvalue of the matrix $A_0$, then $\sR_\lambda(A_0)=\{0\}$
and (i)  together with Proposition~\ref{Lem:SingChainMani} imply
\begin{equation*}
    W^{-1} \left( \{0\}^{n_0} \times \{0\}^{|\alpha|}
\times \mathbb C^{|\veps|-n_\veps} \times \{0\}^{|\eta|}\right)
= \sR_c(\cA) \subset \sR_\lambda(\cA).
\end{equation*}
 so that \eqref{eq:Rc_Weimar} holds in this case.
Next assume that $\lambda\in\sigma_p(A_0)$ and it clearly suffices to show
\begin{equation}\label{eq:Rc_Weimar+}
    W^{-1} \left( \sR_\lambda(A_0) \times \{0\}^{|\alpha|}
\times \mathbb \{0\}^{|\veps|-n_\veps} \times \{0\}^{|\eta|}\right)
 \subset \sR_\lambda(\cA).
\end{equation}
The matrix $A_0$ is in Jordan canonical form and it is no restriction to assume
that the first Jordan block in $A_0$ is of the form
$J_n(\lambda) = \lambda I_n + N_n$ for some $n\in\N$.
Denote by $e_i$, $i=1,\ldots,n_0$, the
standard unit vectors in $\mathbb C^{n_0}$. Then, for $i=1,\ldots,n$,
\begin{equation}\label{frijoles}
     A_0e_i =  \begin{pmatrix}
J_{n}(\lambda)&0\\
0&*\end{pmatrix}
e_i =
\left\{
\begin{array}{ll}
 \lambda e_i ,& \mbox{if } i=n,\\[1ex]
 e_{i+1} +\lambda e_i,& \mbox{if } i=1,\ldots,n-1,
\end{array}
\right.
\end{equation}
where the symbol $*$ stands for possibly more Jordan blocks. Set
\begin{equation*}
x_i:= W^{-1} \begin{pmatrix} e_i\\ 0\\ 0\\ 0\end{pmatrix},\quad i=1,\ldots,n.
\end{equation*}
Together with \eqref{frijoles} this implies that for $i=1,\ldots,n-1$
\[
\begin{split}
\begin{pmatrix}   x_{i}\\  x_{i+1}+\lambda x_{i} \end{pmatrix} &=
\left( \begin{array}{l}
W^{-1}
\begin{pmatrix} I_{n_0}&0&0&0\\ 0&N_\alpha &0&0\\
0&0&K_\veps&0\\ 0&0&0&K_\eta^\top\end{pmatrix}\\ \ \\
 W^{-1} \begin{pmatrix} A_0
&0&0&0\\ 0&I_{|\alpha|}&0&0\\ 0&0&L_\veps&0\\
0&0&0&L_\eta^\top\end{pmatrix}
\end{array}\right)
\begin{pmatrix} e_i\\ 0\\
0\\ 0\end{pmatrix} \\
&=
\begin{pmatrix} F\\ G\end{pmatrix} T
\begin{pmatrix} e_i\\ 0\\ 0\\ 0\end{pmatrix}
\in \ran \begin{pmatrix} F\\ G\end{pmatrix},
\end{split}
\]
which shows with \eqref{penci} that $(x_{i},  x_{i+1}+\lambda x_{i}) \in \cA$.
A similar equation with \eqref{frijoles} for $i=n$ shows that
$(x_n,\lambda x_n)\in  \cA$, and therefore
\begin{equation*}
 (x_1,x_{2}+\lambda x_1 ),
 (x_{2},x_{3}+\lambda x_{3} ),
 \dots,
 (x_{n-1}, x_{n}+\lambda x_{n-1}),
 (x_{n},\lambda x_{n} ) \in \cA.
\end{equation*}
The same arguments work for the remaining Jordan blocks of $A_0$
and \eqref{eq:Rc_Weimar+} and, hence,  \eqref{eq:Rc_Weimar}  is proved.
The inclusions \eqref{eq:Rc_Weimar} and \eqref{eq:Rc_Weimar-} confirm (iii).

Statement (iv) follows from the representations of the singular chain space in~(i),
the root space $\sR_\infty(\cA)$ in~(ii), the root space $\sR_\lambda(\cA)$ in~(iii), and
Definition~\ref{Proper!!!}.
\end{proof}

\end{document}